\documentclass[brochure,12pt]{bourbaki}

\usepackage{natbib}
\makeatletter
\renewcommand*{\NAT@nmfmt}[1]{\textsc{#1}}
\makeatother

\usepackage{url}
\usepackage[utf8]{inputenc}
\usepackage[T1]{fontenc}
\usepackage[french]{babel}
\iftrue
\usepackage{lmodern}
\usepackage{amssymb,mathrsfs}
\else\iffalse
\usepackage[full]{textcomp}
\usepackage[sups,osf]{fbb} 
\usepackage[scaled=.95]{cabin} 
\usepackage[varqu,varl]{inconsolata} 
\usepackage[libertine,bigdelims,vvarbb]{newtxmath} 
\usepackage[scr=boondoxo]{mathalfa} 
\else
\usepackage[full]{textcomp} 
\usepackage[p,osf]{fbb} 
\usepackage[scaled=.95,type1]{cabin} 
\usepackage[varqu,varl]{zi4}
\usepackage[T1]{fontenc} 
\usepackage[libertine]{newtxmath}
\usepackage[scr=boondoxo,bb=boondox,frak=boondox]{mathalfa}
\fi\fi

\date{Mars 2018}
\bbkannee{70\textsuperscript{e} ann\'ee, 2017--2018}
\bbknumero{1144}
\title{Relations de Hodge--Riemann \\ et combinatoire des matroïdes}
\subtitle{d'apr\`es K.~Adiprasito, J.~Huh et E.~Katz}
\author{Antoine CHAMBERT--LOIR}
\address{
Université Paris Diderot, Sorbonne Université, CNRS, \\
Institut de Mathématiques de Jussieu-Paris Rive Gauche, IMJ-PRG, \\
F-75013, Paris, France.}
\email{antoine.chambert-loir@math.univ-paris-diderot.fr}


\usepackage{aclraccourcis}

\begin{altabstract}
Finite matroids are combinatorial structures that express
the concept of linear independence.
In 1964, G.-C.~Rota conjectured that the coefficients of
the ``characteristic polynomial'' of a matroid~$M$,
polynomial whose coefficients enumerate its subsets of given rank,
form a log-concave sequence.
K.~Adiprasito, J.~Huh et E.~Katz just proved this conjecture
using methods which, although entirely combinatorial,
are inspired by algebraic geometry.
From the Bergman fan of the matroid~$M$, they define
a grader ``Chow ring'' $A(M)$ for which they prove analogs
of the Poincaré duality, the Hard Lefschetz theorem,
and the Hodge--Riemann relations. The sought for log-concavity inequalities
are then analogous to the Khovanskii--Teissier inequalities.
\end{altabstract}

\begin{abstract}
Les matroïdes finis sont des structures combinatoires qui expriment
la notion d'indépendance linéaire.
En 1964, G.-C.~Rota conjectura que les
coefficients du « polynôme caractéristique » d'un matroïde~$M$,
polynôme dont les coefficients énumèrent ses sous-ensembles de rang donné,
forment une suite log-concave.
K.~Adiprasito, J.~Huh et E.~Katz viennent de démontrer
cette conjecture par des méthodes qui, bien qu'entièrement
combinatoires, sont inspirées par la géométrie algébrique.
À partir de l'éventail de Bergman du matroïde $M$, ils définissent
en effet un « anneau de Chow » gradué $A(M)$ pour lequel
ils établissent des analogues de la dualité de Poincaré, du théorème de Lefschetz difficile
et des relations de Hodge--Riemann. Les inégalités de log-concavité recherchées
sont alors analogues aux inégalités de Khovanskii--Teissier.
\end{abstract}

\usepackage[hypertexnames]{hyperref}

\begin{document}
\maketitle

\bgroup
\itshape
\raggedright
\advance\leftskip .42\textwidth
... alors que le savant classe hors du temps vécu, situe et fixe à l'écart
de la vie. \par
\raggedleft\normalfont
Philippe Jaccottet, \emph{La semaison}\par
\egroup

\section{Matroïdes}
Inventée par~\cite{whitney1935}, la notion de \emph{matroïde}
est une abstraction de la propriété d'indépendance linéaire.
Elle admet plusieurs formalisations équivalentes,
en termes de \emph{bases}, de \emph{parties libres},
de \emph{circuits}, d'une \emph{fonction rang}, etc.
Nous utiliserons ici celle qui met en jeu le concept de \emph{plat}.

\begin{defi}
Soit $E$ un ensemble \emph{fini}.\footnote{%
Dans ce texte, il sera uniquement question de matroïdes finis.
En fait, la formalisation des matroïdes sur un ensemble infini
est récente:  une condition naturelle consiste à imposer
à la propriété d'être une partie libre d'être de caractère fini,
mais \citet{bruhn-diestel-kriesell-pendavingh-wollan2013}
mettent en évidence qu'on obtient une classe plus intéressante
de matroïdes infinis en imposant un axiome d'existence
de parties libres maximales.}
Un matroïde~$M$ sur~$E$
est la donnée d'une partie~$\mathscr P_M$ 
de $\mathfrak P(E)$ --- les \emph{plats} de~$M$
--- vérifiant les propriétés suivantes:
\begin{enumerate}
\def\theenumi{\roman{enumi}}\def\labelenumi{(\theenumi)}
\item L'intersection de toute famille de plats de~$M$ est un plat.
\item Pour tout plat~$P$ de~$M$, distinct de~$M$,
l'ensemble des plats de~$M$ qui sont minimaux 
parmi ceux contenant strictement~$P$ recouvre~$E$.
\end{enumerate}
\end{defi}

On notera en général $\abs M$, voire $M$,
l'ensemble sous-jacent à un matroïde~$M$.

\subsection{}
Rappelons brièvement 
les autres formalisations de la notion de matroïde.
en renvoyant à~\citep{welsh1976}, \citep{white1986,white1987} 
ou \citep{oxley1992}
pour plus de détails.

Soit $M$ un matroïde sur un ensemble~$E$.
Muni de la relation d'inclusion sur~$\mathfrak P(E)$,
l'ensemble des plats de~$M$ est un ensemble ordonné.
C'est un \emph{treillis} : 
toute partie possède une borne inférieure (notée~$\wedge$), 
l'intersection de ses membres,
et une borne supérieure (notée~$\vee$),
l'intersection de la famille des plats de~$M$
qui contiennent chacun de ses membres. 
Si $X$ est une partie de~$E$, on note $\langle X\rangle$
le plat engendré par~$X$, 
c'est-à-dire le plus petit plat de~$M$ qui contient~$X$.
Au matroïde~$M$, on associe naturellement deux fonctions
\emph{rang} et \emph{corang}: le rang~$\rk_M(P)$ d'un plat~$P$ est la plus
grande longueur d'une chaîne de plats de sommet~$P$,
son corang~$\cork_M(P)$ 
est la plus grande longueur d'une chaîne de plats de
base~$P$.
On définit plus généralement le rang (resp. le corang) 
d'une partie~$A$ de~$M$ comme celui du plat qu'elle engendre.

Une partie~$L$ de~$E$ est dite \emph{liée}
s'il existe une partie $L'\subsetneq L$ 
telle que $\langle L'\rangle=\langle L\rangle$;
elle est dite \emph{libre} sinon.
L'ensemble des parties libres vérifient les propriétés suivantes:
\begin{enumerate}
\def\theenumi{\roman{enumi}}\def\labelenumi{(\theenumi)}
\item $\emptyset$ est libre;
\item Toute partie d'une partie libre est libre;
\item Si $A$ et $B$ sont des parties libres de $M$
telles que $\Card(B)>\Card(A)$, il existe
$x\in B\setminus A$ tel que $A\cup\{x\}$ soit libre
(variante du « lemme d'échange »).
\end{enumerate}
Inversement, toute partie de~$\mathfrak P(E)$ vérifiant ces
trois propriétés est l'ensemble des parties libres d'un unique matroïde sur~$E$.

Une \emph{base} de~$M$ est une partie libre maximale.
Il en existe; on déduit du lemme d'échange que toutes les bases
de~$M$ ont même cardinal et que toutes les chaînes maximales
de plats ont même cardinal, 
égal au rang de~$M$ (c'est-à-dire du plat~$E$ de~$M$).
Autrement dit, pour tout plat~$P$ de~$M$, 
on a $\rk_M(P)+\cork_M(P)=\rk_M(\abs M)$
(le treillis associé à~$M$ est « caténaire »).
De plus, pour tout couple $(P,Q)$ de plats de~$M$, on a la
relation
\begin{equation}
\rk_M(P)+\rk_M(Q) \geq \rk_M(P\wedge Q)+\rk_M(P\vee Q)\ ;
\end{equation}
le treillis associé à~$M$ est dit \emph{sous-modulaire}.
Inversement, tout treillis caténaire sous-modulaire
est le treillis des plats d'un matroïde.

\begin{exem}
Soit $V$ un espace vectoriel (resp. un espace affine,
resp. un espace projectif) sur un corps~$K$
et soit $\Phi=(v_e)_{e\in E}$ une famille finie d'éléments de~$V$.
Il existe un matroïde~$M$ dont
les plats sont les parties de~$E$ de la forme $\{e\in E\sozat v_e\in W\}$,
où $W$ parcourt l'ensemble des sous-espaces vectoriels (resp. affines,
resp. projectifs) de~$V$.
Les parties libres de ce matroïde sont les sous-familles 
libres (resp. affinement libres, resp. projectivement libres) de~$\Phi$.
Dans le cas vectoriel,
le rang d'un plat~$P$ est la dimension de l'espace vectoriel engendré
par la famille $(v_e)_{e\in P}$;
dans le cas affine resp. (projectif), on a $\rk_M(\emptyset)=0$
et $\rk_M(P)-1$ est la dimension du sous-espace affine (resp. projectif)
de~$V$ engendré par la famille~$(v_e)_{e\in P}$.

De tels matroïdes sont dits \emph{représentables} sur~$K$.

Lorsqu'on prend pour $\Phi$ l'ensemble
des points de l'espace projectif $\P_2(\F_2)$, on
obtient le  \emph{matroïde de Fano}. 
Il n'est représentable que sur un corps de caractéristique~$2$.

Les matroïdes représentables apparaissent naturellement 
via l'arrangement d'hyperplans qu'ils définissent dans l'espace vectoriel
dual~$V^\vee$ (resp. l'espace affine, resp. projectif).
Les plats sont alors les intersections d'hyperplans de l'arrangement
et leur rang est leur codimension.

D'après~\citet{nelson2018},
lorsque $n$ tend vers l'infini, la proportion
du nombre de matroïdes sur l'ensemble $\{1,\dotsc,n\}$
qui sont représentables (sur \emph{un} corps non précisé)
tend vers~$0$.
\end{exem}

\begin{exem}
Soit $G$ un graphe fini et soit $E$ l'ensemble des arêtes de~$G$.
Il existe un unique matroïde $M(G)$ sur~$E$ 
dont les parties libres sont les forêts de~$G$.
De manière équivalente,
ses circuits, c'est-à-dire les parties liées minimales,
sont les cycles dans le graphe~$G$;
ses plats sont les parties~$P$
telles que les extrémités de toute arête $e\in E\setminus P$ 
n'appartiennent pas à la même composante
connexe du sous-graphe de~$G$ ayant même ensemble de sommets que~$G$
et~$P$ pour ensemble d'arêtes.

Si $G_P$ est le plus grand sous-graphe de~$G$ d'ensemble
d'arêtes~$P$, $\rk_{M(G)}(P)+\Card(\pi_0(G_P))$
est le nombre de sommets de~$G$.

Le matroïde~$M(G)$ est représentable sur tout corps.
Soit~$S$ l'ensemble des sommets de~$G$.
Soit $K$ un corps et notons $(x_s)$ la base canonique
du $K$-espace vectoriel~$K^S$. 
Fixons une orientation~$F$
de~$G$ et identifions une flèche~$f\in F$ à l'arête correspondante
$\{f,\overline f\}$.
Pour toute flèche~$f\in F$ d'origine~$o$ et de terme~$t$,
posons $v_f=x_t-x_o\in K^S$. Alors le matroïde
associé à la famille~$(v_f)_{f\in F}$ s'identifie au matroïde~$M(G)$.
\end{exem}

\subsection{}
Mentionnons quelques autres constructions de matroïdes.

\def\restr{\mathord{\,|\,}}
\let\del\backslash
\let\contr\slash

\begin{enumerate}\def\theenumi{\alph{enumi}}\def\labelenumi{\theenumi)}
\item
Soit $M_1$ et $M_2$ des matroïdes. 
Il existe alors un unique matroïde $M$ sur 
l'ensemble $\abs {M_1}\coprod\abs{M_2}$
dont les plats
sont les réunions d'un plat de~$M_1$ et d'un plat de~$M_2$.
On le note $M_1\oplus M_2$.

\item
Soit $M$ un matroïde. 
Il existe, sur l'ensemble~$\abs M$, une unique structure de matroïde
dont les bases sont les complémentaires des bases de~$M$.
On l'appelle le \emph{matroïde dual}~$M^*$ de~$M$.
Sa fonction rang est liée à celle de~$M$ par
la relation 
\[ \rk_{M^*}(A)-\rk_{M}(\abs M\setminus A)=\Card(\abs M)-r_M(\abs M),\]
pour toute partie~$A$ de~$\abs M$.

\item
Soit $M$ un matroïde et soit $F$ une partie de~$\abs M$.
L'ensemble des plats de~$M$ qui sont contenus dans~$F$
est une structure de matroïde sur~$F$, que l'on note 
$M\restr F$;
c'est la \emph{restriction} de~$M$ à~$F$;
on la voit aussi comme  
la \emph{suppression} de~$\abs M\setminus F$ dans~$M$
et on la note alors $M\del (\abs M\setminus F)$.
Sa fonction rang est la restriction à~$\mathfrak P(F)$
de la fonction rang de~$M$.

\item
Soit $M$ un matroïde et soit $F$ une partie de~$\abs M$.
L'ensemble des parties~$P$ de~$\abs M\setminus F$
telles que $P\cup F$ soit un plat de~$M$
est une structure de matroïde sur l'ensemble $\abs M\setminus F$;
autrement dit, le treillis de ses plats est le sous-treillis de~$\mathscr P_M$
formé des  plats de~$M$ qui contiennent~$F$.
On l'appelle la \emph{contraction} de~$F$ dans~$M$
et on la note~$M\contr F$.
Sa fonction rang vérifie $\rk_{M/F}(A)=\rk_M(A\cup F)-\rk_M(A)$,
pour toute partie~$A$ de~$\abs M\setminus F$.
\end{enumerate}

\subsection{}
Soit $M$ un matroïde.
Une boucle de~$M$ est un point~$e\in\abs M$ qui appartient à tout plat.
Soit $m=\langle\emptyset\rangle$ le plus petit plat de~$M$.
Le matroïde contracté $M/m$ 
est sans boucle et a même treillis des plats que~$M$.

Supposons $M$ sans boucle.
La relation $\langle x\rangle=\langle y\rangle$ dans~$E$
est une relation d'équivalence dont les classes
d'équivalence sont les plats de~$M$ de rang~$1$. Lorsque ces classes
d'équivalence sont réduites à un élément,
on dit que le matroïde est une \emph{géométrie combinatoire.}
En général, le matroïde~$M$ induit
une géométrie combinatoire~$\overline M$ sur l'ensemble quotient~$\overline E$;
la surjection canonique de~$E$ sur~$\overline E$
induit un isomorphisme du treillis des plats de~$M$ sur celui de~$\overline M$.

\begin{defi}
Soit $M$ un matroïde.
Le \emph{polynôme caractéristique} de~$M$ est défini par
\[ \chi_M(T) = \sum_{A\subset \abs M} (-1)^{\Card(A)} T^{\cork_M(\langle A\rangle)}.
\]
\end{defi}

C'est un polynôme unitaire de degré~$\leq\rk(M)$. 
Lorsque $\abs M$ est vide, on a $\chi_M(T)=1$.
Il y a deux matroïdes sur un ensemble~$\{e\}$ de cardinal~$1$:
si $\mathscr P_M=\{\emptyset, \abs M\}$, on a $\rk(M)=1$ et $\chi_M(T)=T-1$;
si $\mathscr P_M=\{\abs M\}$, on a $\rk(M)=0$ et $\chi_M(T)=0$.
En général, le polynôme~$\chi_M(T)$ se calcule par récurrence à partir 
des deux règles :
\[ \chi_{M_1\oplus M_2}(T) = \chi_{M_1}(T) \chi_{M_2}(T)  \]
et
\[ \chi_M(T) = \chi_{M\backslash e}(T) - \chi_{M/e}(T) \]
pour tout point $e\in \abs M$ qui n'appartient pas à toute base de~$M$.
(Ces deux règles généralisent celles qui régissent le polynôme chromatique d'un graphe.)
 
Si $\abs M$ n'est pas vide, on a $\chi_M(1)=0$; 
on définit alors le polynôme caractéristique réduit par :
\[ \overline{\chi_M}(T) = \chi_M(T)/(T-1). \]

\begin{exem}
Supposons que $M$ soit le matroïde~$M(G)$ associé à un graphe fini~$G$.
Alors, 
pour tout entier~$q$,
\[ \chi_G(T) =  T^{\Card(\pi_0(G))} \chi_M(T) \]
est le polynôme chromatique de~$G$: pour tout entier~$q$,
$\chi_G(q)$ est le nombre de coloriages de l'ensemble des sommets
de~$G$ avec~$q$ couleurs tels que deux sommets reliés par une arête
soient de couleurs distinctes.
\end{exem}

\begin{exem}
Soit $K$ un corps, soit $n$ un entier, 
soit $(v_1,\dotsc,v_r)$ une famille libre de~$K^n$ 
et soit $V$ le sous-espace vectoriel de~$K^n$ qu'elle engendre.
Notons~$M$ le matroïde représentable correspondant.
Le polynôme caractéristique de~$M$ s'interprète géométriquement 
dans l'anneau de Grothendieck $K_0(\mathrm{Var}_K)$ des $K$-variétés.

Rappelons que cet anneau est défini comme le quotient du groupe
abélien libre sur l'ensemble $\mathrm{Var}_K$
des classes d'isomorphie de $K$-variétés (=~$K$-schémas de type fini)
par la relation de découpage $[X]=[X\setminus Y]+[Y]$ si $X$
est une $K$-variété et $Y$ un fermé de~$X$, muni du produit
$[X][Y]=[X\times_K Y]$. Si $X$ est une $K$-variété, on note
$\operatorname e(X)$ sa classe dans $K_0(\mathrm{Var}_K)$.
L'application~$\operatorname e$ 
est une caractéristique d'Euler universelle. En particulier,
lorsque $K=\C$, l'application qui, à une $\C$-variété~$V$,
associe son polynôme de Hodge--Deligne $E_V(u,v)$, se factorise par
un homomorphisme d'anneaux de~$K_0(\mathrm{Var}_\C)$ dans~$\Z[u,v]$.
De même, lorsque $K$ est un corps fini, l'application qui, à une
$K$-variété~$V$, associe le cardinal de~$V(K)$, se factorise
par un homomorphisme d'anneaux de~$K_0(\mathrm{Var}_K)$ dans~$\Z$.

Notons $\mathbf L=\operatorname e(\mathbf A^1_K)$ 
la classe de la droite affine.
Dans l'anneau de Grothendieck $K_0(\mathrm{Var}_K)$ des $K$-variétés,
on a alors la relation
\[ \operatorname e(V \cap \mathbf G_{\mathrm m,K}^n) = \chi_M(\mathbf L). \]
(Cela se déduit de la formule d'inversion de Möbius dans 
le treillis des plats du matroïde~$M$ et de la formule~\eqref{eq.chiM};
voir plus bas.)
Comme l'unique homomorphisme
d'anneaux de~$\Z[T]$ dans~$K_0(\mathrm{Var}_K)$ qui applique~$T$
sur~$\mathbf L$ est injectif, cette relation caractérise~$\chi_M$.

Lorsque $K=\C$, le polynôme de Hodge--Deligne de la variété
quasi-projective $V\cap (\C^\times)^n$ est donc égal à $\chi_M(uv)$.
Lorsque $K=\F_q$ est un corps fini de cardinal~$q$, on a  de même
$\Card(V\cap (K^\times)^n)=\chi_M(q)$.

En passant au quotient par l'action par homothéties de~$\mathbf G_{\mathrm m,K}$
sur l'espace affine~$\mathbf A^n_K$, on en déduit aussi 
une interprétation géométrique 
du polynôme caractéristique réduit de~$M$:
\[ \operatorname e(\P(V\cap \mathbf G_{\mathrm m,K}^n))= \overline{\chi_M}(\mathbf L). \]
\end{exem}

Voici le théorème principal de cet exposé:
\begin{theo}[\citealt*{adiprasito-huh-katz2015}]
\label{theo.ahk}
Soit $M$ un matroïde de \mbox{rang~$>0$;} posons $r=\rk(M)-1$ et
définissons des nombres entiers $\mu^k(M)$, pour $0\leq k\leq r$, par
\[ \overline{\chi_M}(T) = \sum_{k=0}^r (-1)^k \mu^k(M) T^{r-k}. \]
Alors, la suite $(\mu^0(M),\dotsc,\mu^r(M))$ est \emph{log-concave :}
\begin{enumerate}
\def\theenumi{\roman{enumi}}\def\labelenumi{(\theenumi)}
\item Pour tout entier~$k$ tel que $0\leq k\leq r$, on a $\mu^k(M)>0$;
\item Pour tout entier~$k$ tel que $0<k<r$, on a $\mu^{k-1}(M)\mu^{k+1}(M)
\leq \mu^k(M)^2$.

\noindent En particulier, cette suite est \emph{unimodale}:
\item Il existe un entier~$\ell$
tel que 
\[ \mu^0(M)\leq \mu^1(M) \leq\dotsb \leq \mu^{\ell}(M)
 \geq \mu^{\ell+1}(M)\geq \dotsb \geq \mu^r(M). \]
\end{enumerate}
\end{theo}

Plusieurs corollaires de ce théorème avaient été conjecturés
dans les années 1970, 
\citep{rota1971,heron1972,mason1972,welsh1976},
parfois d'abord dans le cas des graphes
et de leurs polynômes chromatiques
\citep{read1968,hoggar1974}:
\begin{coro}[conjecture de Heron--Rota--Welsh]
\label{coro.welsh}
Soit $M$ un matroïde et soit $r=\rk(M)$.
Définissons des nombres entiers $w_k(M)$, pour $0\leq k\leq r$, par
\[ {\chi_M}(T) = \sum_{k=0}^r (-1)^k w_k(M) T^{r-k}. \]
Alors, la suite $(w_0(M),\dotsc,w_r(M))$ est \emph{log-concave}
et \emph{unimodale}.
\end{coro}
Les entiers~$w_k(M)$ sont appelés nombres de Whitney de première
espèce du matroïde~$M$; ils sont positifs ou nuls (voir plus bas).

\begin{coro}[conjecture de Welsh--Mason]
\label{coro.welsh-mason}
Soit $M$ un matroïde et soit $r=\rk(M)$.
Pour tout entier~$k$, notons $f_k(M)$ le nombre de parties libres de~$M$
de cardinal~$k$. La suite $(f_0(M),\dotsc,f_{r}(M))$ est log-concave.
\end{coro}

Suivant~\cite{lenz2013}, ce corollaire se déduit du théorème~\ref{theo.ahk}
en considérant le matroïde $M'$
sur l'ensemble $E'=E\coprod\{e\}$ dont les plats
sont la partie vide et les parties $P\cup\{e\}$, pour tout plat~$P$ de~$M$
(« coextension libre » de~$M$).
Ce matroïde est de rang~$r+1$ et,
d'après~\citep[Remark~6.15.3c]{brylawski1982},
son polynôme caractéristique réduit
vérifie
\[ \overline{\chi_{M'}}(T) = \sum_{k=0}^r (-1)^k f_k(M) T^{r-k}. \]
%

\subsection{}
Bien que de nature purement combinatoire,
la démonstration du théorème~\ref{theo.ahk}
est inspirée par la géométrie algébrique.

Supposons en effet 
que le matroïde~$M$ soit représentable sur un corps~$K$,
associé à l'arrangement d'hyperplans défini
par la trace, sur un sous-espace projectif~$V$ de dimension~$r$,
des hyperplans de coordonnées de~$\P_{n}(K)$.
L'involution~$\iota$ de Cremona est l'automorphisme birationnel
de $\P_{n,K}$
donné par $[x_0:\dotsb:x_n]\mapsto [x_0^{-1}:\dotsb:x_n^{-1}]$.
L'adhérence de Zariski dans~$\P_{n,K}\times\P_{n,K}$ du graphe
de la restriction de~$\iota$ à~$V$ est alors une compactification
lisse~$\tilde V$ de~$V\cap \mathbf G_{\mathrm m,K}^n$ dont le bord
est un diviseur à croisements normaux stricts.
(Cette construction est due à \cite{deconcini-procesi1995a}.)

\begin{theo}[\citealt*{huh-katz2012}]
\label{theo.hk}
On a l'égalité 
\[ [\tilde V] = \sum_{k=0}^r \mu^k(M) [\P_{r-k}\times\P_k] \]
dans le groupe $A_r(\P_n\times\P_n)$
des classes de cycles de dimension~$r$ sur~$\P_n\times\P_n$.
\end{theo}
De manière équivalente, on a
\[ \mu^k(M) = \deg ( c_1(\mathscr L_1)^{r-k} c_1(\mathscr L_2)^k \cap [\tilde V]) ,\]
où $\mathscr L_1$ et $\mathscr L_2$ sont les fibrés en droites
sur~$\P_n\times\P_n$ déduits du fibré en droites~$\mathscr O_{\P_n}(1)$
par la première et la seconde projection.
La log-concavité de la suite $(\mu^0(M),\dotsc,\mu^r(M))$
se déduit alors du théorème de l'indice de Hodge,
sous la forme des \emph{inégalités de Khovanskii-Teissier}
\citep{khovanskii1988,teissier1979}.

Lorsque $K$ est de caractéristique zéro,
\citet{huh2012} avait donné une première démonstration de
la log-concavité de la suite $(\mu^k(M))$,
dans laquelle les coefficients~$\mu^k(M)$ apparaissent
comme les nombres de Milnor de la trace, sur un sous-espace
général de dimension~$k$,
de la réunion des hyperplans de l'arrangement définissant~$M$.
L'inégalité de log-concavité se déduit alors de l'inégalité
de Teissier pour les multiplicités~\citep[Appendice]{eisenbud-levine-teissier1977}.

On trouve par ailleurs dans~\citet{huh2012} un très joli théorème
qui caractérise les classes de cycles dans~$\P_m\times\P_n$
qui sont, à un scalaire près, représentés par une sous-variété
irréductible:
\begin{theo}
Soit $r$ un entier tel que $0\leq r\leq \inf(m,n)$ et soit  
$(a_0,\dotsc,a_r)$ une suite d'entiers relatifs. Pour
que la classe de cycle
\[ \alpha = \sum_{k=0}^r a_k [\P_{r-k}\times\P_k ] \in A_r(\P_m\times\P_n) \]
soit un multiple positif de la classe d'une sous-variété irréductible~$V$
de~$\P_m\times\P_n$, il faut et il suffit que l'une des deux conditions
suivantes soit satisfaite:
\begin{itemize}
\item La classe $\alpha$ est un multiple positif de
$[\P_m\times\P_n]$, $[\P_m\times\P_0]$,
$[\P_0\times\P_n]$ ou $[\P_0\times\P_0]$;
\item La suite $(a_0,\dotsc,a_r)$ est log-concave
et l'ensemble des entiers~$k$ tels que $a_k=0$ est un intervalle.
\end{itemize}
\end{theo}
Autrement dit, les inégalités de Khovanskii-Teissier caractérisent
précisément les classes de cycles effectifs.
Dans un esprit similaire, mentionnons le contre-exemple
de~\citet{babaee-huh2017} à une version de la conjecture de Hodge
pour les courants positifs: un courant fortement positif de type~$(2,2)$
sur une  variété complexe torique de dimension~4, lisse et projective, 
qui n'appartient pas au cône convexe fermé engendré par les
courants d'intégration sur les sous-variétés.

\subsection{}
Ainsi, pour démontrer le théorème~\ref{theo.ahk},
\citet*{adiprasito-huh-katz2015}
associent à une géométrie combinatoire~$M$
une $\R$-algèbre graduée~$A(M)_\R$ qui vérifie
les énoncés analogues à la dualité de Poincaré,
le théorème de Lefschetz difficile et les inégalités
de Hodge--Riemann. 

De fait, $A(M)_\R$ sera l'anneau de Chow  (à coefficients réels)
d'une variété torique lisse~$X_M$ (sur un corps~$K$ arbitraire)
introduite par~\citet{feichtner-yuzvinsky2004}.
En général, la variété $X_M$ n'est pas propre,
et elle est de dimension~$>\rk_M(M)$,
de sorte    que son anneau de Chow $A(M)$,
ou son anneau de cohomologie $H(M)$, n'a aucune raison de se comporter
comme celui d'une variété projective lisse de dimension~$\rk(M)$.

Lorsque le matroïde~$M$ est représentable,
$X_M$ admet une sous-variété projective lisse~$Y_M$
(c'est d'ailleurs la variété~$\tilde V$ du paragraphe précédent)
telle que l'application $z\mapsto z \cap [Y_M]$
induise un isomorphisme de $ A(X_M)$ sur $ A(Y_M)$.
En revanche,
lorsque $M$ n'est pas représentable sur~$K$,
il n'existe pas de $K$-variété
propre et lisse~$Y$ et de morphisme de variétés~$f\colon Y\to X_M$
tel que $f^*$ induise un isomorphisme de~$A(X_M)$ sur~$A(Y)$
 \citep*[th.~5.12]{adiprasito-huh-katz2015}.

\subsection{}
Le principe n'est bien sûr pas nouveau
de puiser l'inspiration d'une preuve de résultats de nature combinatoire
dans la  géométrie algébrique,
par exemple via le dictionnaire qui, à  
tout polytope~$P$ de~$\R^r$, à sommets entiers et de dimension~$r$,
associe une variété torique projective polarisée~$(X_P,L)$.

Dans ce dictionnaire, la dimension 
$h^0(X_P,L^{\otimes n})$  
de l'espace de sections globales de la puissance $n$-ième de~$L$
correspond au nombre de points entiers
du polytope~$nP$ et, via la formule de Hilbert--Samuel,
relie le degré de~$X_P$ au volume de~$P$. 
Plus généralement, les volumes mixtes correspondent
à des nombres d'intersection.
C'est ainsi que \citet{khovanskii1988} et \citet{teissier1979} 
déduisent les inégalités
d'Alexandrov--Fenchel et de Brunn--Minkowski
du théorème de l'indice de Hodge sur les surfaces
et du théorème de Bertini.

Citons ainsi la conjecture de {McMullen} 
décrivant ce que peut être le nombre~$f_i$ de 
faces de dimension~$i$ donnée d'un polytope \emph{simple}~$P$ de
dimension~$d$. 
McMullen considère une suite $(h_0,\dotsc,h_d)$ obtenue par
combinaisons linéaires adéquates des~$f_i$ et postule trois familles
de conditions sur cette suite pour que la suite initiale $(f_0,\dotsc,f_d)$
soit la suite des nombres de faces d'un polytope simple;
les premières, $h_k=h_{d-k}$, sont 
les \emph{relations de Dehn--Sommerville;}
la seconde famille s'écrit $h_0\leq h_1\leq\dotsb \leq h_{\lfloor d/2\rfloor}$;
la troisième est un peu technique et
ce n'est pas la peine de la recopier ici.

\cite{billera-lee1980} ont prouvé qu'elles sont suffisantes,
et leur nécessité est démontrée par \cite{stanley1980b}
à l'aide de la géométrie algébrique. 
Lorsque le polytope est à sommets
entiers, il observe en effet
que l'entier~$h_k$ est la dimension de l'espace de cohomologie
$H^k(X_Q)_\R$ 
de la variété torique~$X_Q$ (sur~$\C$, disons) 
associée au polytope~$Q$ polaire de~$P$,
de sorte que les conditions de McMullen découlent respectivement de trois 
propriétés cohomologiques de la cohomologie de~$X_Q$:
la dualité de Poincaré,
le théorème de Lefschetz difficile,
et le fait que l'algèbre de cohomologie $H(Y_Q)_\Q$ à coefficients
rationnels soit engendrée par $H^2(Y_Q)$.
La variété~$Y_Q$ est projective, mais n'est pas nécessairement lisse,
mais l'hypothèse que le polytope~$P$ est simple assure que la variété~$Y_Q$
est une « {orbifolde} », c'est-à-dire localement
le quotient d'une variété lisse par l'action d'un groupe fini, 
de sorte que ces énoncés cohomologiques restent valides.

Par homothéties et approximation, on peut parfois
déduire du cas d'un polytope à sommets entiers
le cas d'un polytope à sommets réels quelconques,
mais il existe des polytopes dont la combinatoire ne peut 
pas être obtenue par une déformation en un polytope à sommets rationnels.
Cette complication ne se produit pas dans les deux exemples précédents;
pour la conjecture de McMullen, l'hypothèse que le polytope
est simple est essentielle.

Ultérieurement, \cite{stanley1987} a généralisé cette
étude aux polytopes non nécessairement simples:
lorsque $P$ est à sommets entiers,
l'entier~$h_k$ est alors la dimension de l'espace de cohomologie
d'intersection $IH^k(X_Q)_\R$.
L'extension de cette relation aux polytopes à sommets arbitraires
a motivé  d'une part
une approche combinatoire de~\citet{mcmullen1993},
et d'autre part le développement d'une « cohomologie d'intersection
des polytopes »  et la preuve du théorème de Lefschetz difficile
dans ce contexte \citep{karu2004}.

\subsection{}
Le titre de ce rapport mentionne les inégalités de Hodge--Riemann.
Comme on le verra plus tard, ces inégalités renforcent le théorème
de Lefschetz difficile: si ce dernier signifie  qu'une certaine
forme bilinéaire est non dégénérée, les inégalités de Hodge--Riemann
en précisent la signature.

En géométrie kählérienne 
le théorème de Lefschetz difficile est démontré \emph{avant}
les inégalités de Hodge--Riemann.
C'est a fortiori le cas en géométrie algébrique,
en particulier sur un corps de caractéristique positive où
les inégalités de Hodge--Riemann sont encore une conjecture.
 
En revanche, il semble que les approches combinatoires
du théorème de Lefschetz difficile ne puissent faire l'économie
des inégalités de Hodge--Riemann.
Outre le théorème de~\citet{karu2004} déjà mentionné,
mentionnons le rôle crucial
que jouent ces inégalités dans la preuve du théorème
de décomposition que proposent~\citet{decataldo-migliorini2005}
(voir aussi l'exposé de~\citet{williamson2017} dans ce séminaire).
Citons enfin la preuve par~\citet{elias-williamson2014}
de la positivité des coefficients des polynômes de Kazhdan--Lusztig
associés à un système de Coxeter général;
cf. également le rapport de~\citet{riche2018} dans ce séminaire.

Les travaux dont il est question dans ce rapport
suggèrent l'intérêt d'un analogue de la
théorie de Hodge en géométrie tropicale.
Je me contente ici de renvoyer à
l'article de~\citet*{itenberg-katzarkov-mikhalkin-zharkov2016}
où sont   construits des espaces de $(p,q)$-formes sur
les variétés tropicales.

\subsection{}
L'unimodalité et la log-concavité sont des thèmes prégnants
de la combinatoire énumérative.
L'article de~\cite{stanley1989}, complété par celui de~\cite{brenti1994},
en fournit une excellente introduction.

À propos des résultats de cet exposé, 
je renvoie aussi au bref survol \citep*{adiprasito-huh-katz2017}
et surtout à l'article d'exposition de~\citet*{baker2018}.

Je remercie enfin Karim Adiprasito,
Matt Baker, Michel Brion, 
Antoine Ducros, Javier Fresán, Olivier Guichard, June Huh, Ilia Itenberg
et Bernard Teissier 
pour leurs commentaires sur des premières versions de ce rapport.

\section{Éventails}

\subsection{}
Soit $(\mathscr P,\preceq)$ un ensemble ordonné, disons fini,
et soit $A$ un anneau commutatif. L'algèbre de convolution
$\mathscr A(\mathscr P;A)$ est le $A$-module
des fonctions à valeurs dans~$A$ sur l'ensemble des couples $(x,y)$
d'éléments de~$\mathscr P$ tels que $x\preceq y$,
muni du produit de convolution défini par 
\[ \phi* \psi (x,y) = \sum_{x\preceq z\preceq y} \phi(x,z) \psi(z,y). \]
Son élément unité~$\delta$ est l'indicatrice de Kronecker.
Un élément de $\mathscr A(\mathscr P;A)$ est inversible si et seulement 
s'il ne prend que des valeurs inversibles en les couples de la forme~$(x,x)$.
La fonction de Möbius de~$\mathscr P$, notée~$\mu$,
est l'inverse de la fonction constante~$\mathbf 1$ de valeur~$1$;
elle est caractérisée par les relations
\begin{gather}
\mu(x,x)=1  \\
\sum_{x\preceq z\preceq y} \mu(x,z) =0
\end{gather}
pour tout $x\in\mathscr P$ d'une part, 
et tout couple $(x,y)$ d'éléments de~$\mathscr P$ tels que $x\prec y$.
L'algèbre $\mathscr A(\mathscr P;A)$ agit à gauche sur le $A$-module
$\mathscr F(\mathscr P;A)$ des fonctions  de~$\mathscr P$ dans~$A$, 
par la formule:
\[ \phi * f(x) = \sum_{x\preceq y} \phi(x,y) f(y). \]
La \emph{formule d'inversion de Möbius} est alors,
pour deux éléments $f,g$ de~$\mathscr F(P;A)$ 
l'équivalence entre les relations $g=\mathbf 1*f$  et $f=\mu*g$;
autrement dit:
\[ g(x) = \sum_{x\preceq y} f(y) \quad\Leftrightarrow\quad
   f(y) = \sum_{y\preceq x} \mu(y,x) g(x) .
\] 

Lorsque, de plus, $\mathscr P$ est un treillis, on a la relation
(\emph{théorème de Weisner}):
\begin{equation}
\sum_{\sup(x,a)=\sup(\mathscr P)} \mu(\inf(\mathscr P),x) = 0
\end{equation}
pour tout $a\in\mathscr P$ distinct de $\inf(\mathscr P)$.
Et si, de plus, ce treillis est sous-modulaire, le signe 
de la fonction de Möbius est donné par :
\begin{equation}\label{eq.rk-rota}
(-1)^{\rk(y)-\rk(x)} \mu(x,y) \geq 0
\end{equation}
pour $x,y\in\mathscr P$ tels que $x\preceq y$.

\subsection{}
Le polynôme caractéristique d'un matroïde~$M$ s'exprime en termes
de son treillis des plats, par la relation :
\begin{equation}\label{eq.chiM}
 \chi_M(T) = \sum_{P\in\mathscr P_M} \mu(\emptyset,P) T^{\cork_M(P)}. 
\end{equation}
Cette relation permet aussi de supposer,
dans les questions relatives au polynôme caractéristique~$\chi_M$,
que le matroïde~$M$ est une géométrie combinatoire.
Jointe aux inégalités~\eqref{eq.rk-rota},
elle prouve enfin que les entiers~$w_k(M)$
introduits dans le corollaire~\ref{coro.welsh}
sont positifs ou nuls.

\subsection{}
Soit $M$ une géométrie combinatoire de rang~$\geq 1$ sur un ensemble~$E$;
posons $r=\rk_M(M)-1$ et $n=\Card(E)-1$.
On note $N\simeq \Z^n$ le quotient du $\Z$-module libre~$\Z^E$,
de base canonique $(e_i)_{i\in E}$
par le sous-module engendré par $\sum e_i$; pour $i\in E$,
on note encore  $e_i$ l'image dans~$N$ de l'élément correspondant de~$\Z^E$.
Pour toute partie~$I$ de~$E$, on pose $e_I=\sum_{i\in I} e_i$.

Dans tout ce texte, un \emph{drapeau} de plats de~$M$
sera une famille totalement ordonnée de plats de~$M$
distincts de~$\emptyset$ et~$\abs M$. 
Si $\mathscr D$ est un drapeau,
on note alors~$\sigma_{\mathscr D}$ le cône de~$N_\R$
engendré par les vecteurs $e_{P}$, pour $P\in\mathscr D$;
il est de dimension~$\Card(\mathscr D)$.
Notons~$\Sigma_M$ l'ensemble de ces cônes:
c'est l'\emph{éventail de Bergman} du matroïde~$M$
\citep{ardila-klivans2006}.

\begin{prop}
L'ensemble $\Sigma_M$ est un éventail unimodulaire de~$N_\R$,
purement de dimension~$r$.
\end{prop}

Que $\Sigma_M$ soit un \emph{éventail} signifie qu'il n'est pas vide,
que toute face d'un cône de~$\Sigma_M$ appartient à~$\Sigma_M$
et que l'intersection de deux cônes de $\Sigma_M$,
est une face commune de ces deux cônes;
qu'il soit \emph{unimodulaire} signifie que les cônes de~$\Sigma_M$
sont engendrés par une partie d'une base de~$N$;
qu'il soit \emph{purement de dimension~$r$} signifie que 
tout cône maximal de~$\Sigma_M$ est de dimension~$r$.

À tout éventail~$\Sigma$
est classiquement associée une \emph{$K$-variété torique} $X_\Sigma$,
aussi notée~$X(\Sigma)$,
obtenue en recollant les variétés 
affines $X_\sigma = \Spec( K[N^\vee \cap \sigma^\circ])$,
pour $\sigma$ parcourant~$\Sigma$,
où $\sigma^\circ$ désigne l'ensemble des formes linéaires
sur~$N_\R$ qui sont positives en tout point de~$\sigma$;
je renvoie par exemple à~\citep{fulton1993} pour plus de détails
sur cette théorie.
Noter que si $\tau$ est une face de~$\sigma$,
$X_\tau$ est un ouvert de~$X_\sigma$ pour la topologie de Zariski.

Soit $\sigma$ un cône de~$\Sigma$. Si $\sigma$ est engendré
par une partie d'une base de~$N$, la variété~$X_\sigma$ 
est un ouvert de~$\A^E$, isomorphe
à $\A^{\dim(\sigma)} \times \mathbf G_{\mathrm m}^{n-\dim(\sigma)}$.
Si l'éventail~$\Sigma$ est unimodulaire,
cette hypothèse est donc vérifiée pour tout cône,
de sorte que la variété torique~$X_\Sigma$ est lisse.

Nous noterons $X_M$ la variété torique lisse associée à l'éventail~$\Sigma_M$.

\begin{rema}
Soit $K$ un corps.
Supposons que le matroïde~$M$ soit associé à l'arrangement
d'hyperplans d'un sous-espace projectif~$V$ de~$\P_{n}$
de dimension~$r$ découpé par les hyperplans de coordonnées
de~$\P_{n}$ (sur le corps~$K$).
Le support~$\abs{\Sigma_M}$ de l'éventail~$\Sigma_M$,
réunion des cônes de~$\Sigma_M$,
s'interprète alors comme la \emph{tropicalisation}
de $V\cap\mathbf G_{\mathrm m}^n$. 

Il y a plusieurs façons de définir cette tropicalisation.

Munissons le corps~$K$ de la valeur absolue triviale
et considérons le tore $K$-analytique $(\gm^n)^\an$, 
au sens de~\citet{berkovich1990},
c'est-à-dire l'espace
des semi-normes multiplicatives sur la $K$-algèbre
$K[T_1^{\pm1},\dotsc,T_n^{\pm 1}]$  qui sont triviales sur~$K$.
Il est muni d'une application continue et propre de tropicalisation,
$\tau\colon (\gm^n)^\an\to\R^n$
qui applique une semi-norme $\abs{\cdot}$
sur $(\log(\abs{T_1}),\dotsc,\log(\abs{T_n}))$.
L'espace de Berkovich $(V\cap\gm^n)^\an$
de~$V\cap\gm^n$ est un sous-espace de~$(\gm^n)^\an$
et son image par~$\tau$ est égale à~$\abs{\Sigma_M}$.

Peut-être plus élémentairement \citep*[voir][]{einsiedler-kapranov-lind2006},
considérons une extension
algébriquement close~$L$ de~$K$ munie d'une valeur absolue
non archimédienne non triviale, mais triviale sur~$K$,
par exemple une clôture algébrique du
corps $K(\!(z)\!)$ des séries de Laurent.
Alors, $\abs{\Sigma_M}$ est l'\emph{adhérence} de l'image
de~$V(L)\cap (L^\times)^n$
par l'application  $(a_1,\dotsc,a_n)\mapsto (\log(\abs{a_1}),\dotsc,\log(\abs{a_n})$ de~$(L^\times)^n$ dans~$\R^n$.

On peut enfin prendre $K=\C$
et considérer, pour tout nombre réel~$\eps>1$,
l'application $\lambda_\eps\colon (\C^\times)^n\to\R^n$
donnée par $\lambda_\eps(z_1,\dotsc,z_n)=(\log_\eps(\abs{z_1}),\dotsc,\log_\eps(\abs{z_n}))$ (logarithme en base~$\eps$).
Lorsque $\eps$ tend vers~$+\infty$, $\lambda_\eps(V\cap (\C^\times)^n)$
converge vers~$\abs{\Sigma_M}$.
\end{rema}

\subsection{}
Soit $N$ un $\Z$-module libre de rang fini,~$n$,
et soit $\Sigma$ un éventail unimodulaire de~$N$;
notons $r$ la borne supérieure (dans~$\N$) 
des dimensions des cônes de~$\Sigma$.
Soit $K$ un corps ; considérons 
la $K$-variété torique $X_\Sigma$ associée à l'éventail~$\Sigma$,
et notons $A(X_\Sigma)$ son anneau de Chow.

On note $V_\Sigma$ l'ensemble des \emph{rayons} de~$\Sigma$,
c'est-à-dire des générateurs primitifs des cônes
de dimension~$1$ de~$\Sigma$. À tout~$v\in V_\sigma$
correspond un diviseur de Cartier irréductible~$D_v$ de~$X_\Sigma$.

Soit $S_\Sigma=\Z[(T_v)_{v\in V_\sigma}]$ 
l'anneau gradué des polynômes  à  coefficients
entiers et à indéterminées dans~$V_\sigma$.
Pour tout entier~$k$,
on note~$S^k_\Sigma$ sa composante homogène de degré~$k$.

Pour tout cône~$\sigma$ de~$\Sigma$, on 
pose $T_\sigma=\prod_{v\in\sigma\cap V_\Sigma} T_e$;
c'est un monôme de degré~$\dim(\sigma)$.
Pour tout entier~$k$, on définit un sous-module 
\[ Z^k(\Sigma)=\bigoplus_{\substack{\sigma\in\Sigma \\ \dim(\sigma)=k}}
    \Z T_\sigma \ ; \]
c'est un sous-module de $S^k_\Sigma$, nul pour $k>r$.
On pose aussi $Z(\Sigma)=\bigoplus_{k\in\N} Z^k(\Sigma)$.

Soit $I_\Sigma$ l'idéal homogène de~$S_\Sigma$ engendré par les monômes
qui n'appartiennent pas à~$Z(\Sigma)$
et $J_\Sigma$ l'idéal de~$S_\Sigma$ engendré par les polynômes linéaires
$ \sum_{v\in V_\sigma} \phi(v) T_v $,
pour $\phi\in N^\vee$. 

\begin{prop}
L'unique homomorphisme d'anneaux de~$S_\Sigma$ dans~$A(X_\Sigma)$
qui applique tout~$v\in V_\sigma$ sur la classe de~$D_v$
dans~$A^1(X_\Sigma)$ est surjectif.
Son noyau est l'idéal $I_\Sigma+J_\Sigma$.
\end{prop}

L'anneau gradué quotient $A(\Sigma)=S_\Sigma/(I_\Sigma+J_\Sigma)$
sera ainsi appelé 
l'\emph{anneau de Chow} de l'éventail~$\Sigma$.
Pour tout entier~$k$, on note $A^k(\Sigma)$ 
sa composante homogène de degré~$k$;
elle est engendrée par $Z^k(\Sigma)$.

Pour $k>\dim(\Sigma)$, on a $A^k(\Sigma)=0$.

\begin{rema}
On note $\PL(\Sigma)$ (resp. $\PP(\Sigma)$)
l'espace vectoriel (resp. la $\R$-algèbre) 
des fonctions de~$\abs\Sigma$ dans~$\R$
dont la restriction à tout cône de~$\Sigma$ est linéaire
(resp. polynomiale);  une telle fonction est continue.
On dit, par abus, que ce sont les fonctions
linéaires (resp. polynomiales)  par morceaux sur~$\abs\Sigma$.

Pour $v\in V_\Sigma$, il existe
une unique fonction $\phi_v\in\PL(\Sigma)$
telle que pour tout $w\in V_\Sigma$,
on a $\phi_v(w)=1$ si $w=v$, et $\phi_v(w)=0$ sinon.
Ces fonctions~$\phi_v$, parfois appelées \emph{fonctions de Courant},
forment une base de~$\PL(\Sigma)$.

D'après~\citep{billera1989},
l'unique homomorphisme de~$\R\otimes S_\Sigma$ dans~$\PP(\Sigma)$
qui applique $T_v$ sur~$\phi_v$, pour tout $v\in V_\Sigma$,
est surjectif; son noyau est engendré par
les monômes $T_{v_1}\dotsm T_{v_d}$, tels que $v_1,\dotsc,v_d$
n'engendrent pas un cône de~$\Sigma$. 
Par suite, $A(\Sigma)_\R$ est le quotient de $\PP(\Sigma)$
par l'idéal engendré par les restrictions à~$\abs\Sigma$
des formes linéaires sur~$N_\R$.
\end{rema}

\begin{lemm}
L'application de~$\PL(\Sigma)$ dans~$A^1(X_\Sigma)_\R$
qui, pour tout $v\in V_\Sigma$, applique~$\phi_v$
sur la classe de~$D_v$, est surjective. 
Son noyau est le sous-espace de~$\PL(\Sigma)$ engendré 
par les restrictions à~$\abs\Sigma$ des formes linéaires
sur~$N_\R$.
\end{lemm}

À une fonction linéaire par morceaux~$\phi\in\PL(\Sigma)$, on 
associe le $\R$-diviseur (invariant par l'action du tore) 
$\div(\phi)=\sum_{v\in V_\Sigma} \phi(v) D_v$.
Ce diviseur est effectif si $\phi$ est positive sur~$\abs\Sigma$.

Plus généralement,
soit $\sigma$ un cône de~$\Sigma$
et soit $V_\sigma\subset X_\sigma$ l'adhérence
de l'orbite du tore qui lui correspond
(on a $\dim(V_\sigma)=\codim(\sigma)$).
On dit que $\phi$ est \emph{convexe} en~$\sigma$
si la classe du diviseur $\div(\phi)|_{V_\sigma}$
sur~$V_\sigma$ est effective.
Cela revient à dire qu'il existe une forme linéaire~$m$
sur~$N_\R$ telle que $\phi=m$  sur~$\sigma$,
et telle que $\phi\geq m$ 
sur l'\emph{étoile} de~$\sigma$,
c'est-à-dire sur le sous-éventail $\star_\Sigma(\sigma)$
de~$\Sigma$ formé des cônes de~$\Sigma$ contenant une face de~$\sigma$.
On dit que $\phi$ est strictement convexe en~$\sigma$
s'il existe une forme linéaire~$m$ sur~$N_\R$
telle que $\phi=m$ sur~$\sigma$
et telle que $\phi(v)>m(v)$ pour tout rayon~$v$ de l'étoile de~$\sigma$
qui n'appartient pas à~$\sigma$.

L'ensemble des fonctions~$\phi\in\PL(\Sigma)$ qui sont convexes
(resp. strictement convexes) 
en tout cône~$\sigma\in\Sigma$ est 
un cône (resp. un cône ouvert) de~$\PL(\Sigma)$;
on l'appelle le cône \emph{nef} (resp. le cône \emph{ample})
de~$\PL(\Sigma)$ et on le note $\mathscr N_\Sigma$
(resp. $\mathscr K_\Sigma$).
Ces cônes contiennent l'image du dual de~$N_\R$,
on désigne des mêmes lettres leurs images dans~$A^1(X_\Sigma)_\R$.

Si $\mathscr K_\Sigma$ n'est pas vide, 
c'est l'intérieur de $\mathscr N_\Sigma$, lequel est son adhérence.
Cela se produit en particulier lorsque~$\Sigma$ est l'éventail
associé à un matroïde ; en effet, $\Sigma$
est alors un sous-éventail de l'éventail normal d'un polytope.

\subsection{}
Supposons que $\Sigma$ soit l'éventail~$\Sigma_M$
 associé au matroïde~$M$. On notera alors $A(M)$, $S_M$, etc.
les objets $A(\Sigma_M)$, $S_{\Sigma_M}$, etc.

Dans ce cas, les rayons de~$\Sigma_M$
sont les vecteurs~$e_P$, où $P$ parcourt l'ensemble~$\mathscr P^*_M$ des plats de~$M$ tels que $P\neq\emptyset$ et $P\neq \abs M$.
On a donc $S_M=\Z[(T_P)_{P\in\mathscr P^*_M}]$.
L'idéal~$I_M$ est engendré par les monômes quadratiques
$T_P T_Q$, où $(P,Q)$ parcourt l'ensemble des couples d'éléments
incomparables de~$\mathscr P^*_M$,
tandis que l'idéal~$J_M$ est engendré par les polynômes linéaires
\[ \sum_{P\ni i} T_P - \sum_{P\ni j} T_P, \]
où $(i,j)$ parcourt l'ensemble des couples d'éléments de~$\abs M$.

On a $A^k(M)=0$ pour $k>r$.

On définit deux éléments $\alpha_M$ et $\beta_M$ de $A^1(M)$ par 
\[  \alpha_M = \sum_{P\ni i} T_P, \qquad \beta_M =\sum_{P\not\ni i} T_P, \]
où $i$ est un élément arbitraire de~$\abs M$ (ils n'en dépendent pas).
Si $\mathscr D$ est un drapeau de plats donné,
on voit, 
en choisissant~$i$ hors de $\sup(\mathscr D)$, resp. dans $\inf(\mathscr D)$,
que leurs images dans $A^1(M)_\R$ 
appartiennent au cône nef~$\mathscr N_M$.

\begin{prop}\phantomsection\label{prop.deg-muk}
\begin{enumerate}\def\theenumi{\alph{enumi}}\def\labelenumi{\theenumi)}
\item
Il existe un unique homomorphisme de groupes
\[ \deg\colon A^r(M)\to\Z \]
tel que $\deg(T_{P_1}\dotsm T_{P_r})=1$ pour toute suite
$(P_1,\dotsc,P_r)$ de plats de~$M$ vérifiant
$\emptyset \subsetneq P_1\subsetneq\dotsb\subsetneq P_r\subsetneq \abs M$.
C'est un isomorphisme.
De plus, $\deg(\alpha_M^r)=1$.

\item
Pour tout entier $k$ tel que $0\leq k\leq r$, on a
\[ \mu^k(M) = \deg(\alpha_M^{r-k}\beta_M^k). \]
\end{enumerate}
\end{prop}
La  deuxième partie de la proposition est
la généralisation, dans le contexte combinatoire, 
du théorème~\ref{theo.hk}.

\subsection{}
Soit $A$ une $\Z$-algèbre commutative unifère, graduée, artinienne, 
et soit $r$ un entier; on suppose que $A^k=0$ pour $k<0$ ou $k>r$
et on se donne un homomorphisme $\deg\colon A^r\to\Z$.

On dit que $(A,\deg)$ vérifie la \emph{dualité de Poincaré}
si pour tout entier~$k$ tel que $0\leq k\leq r$,
l'application 
$a\mapsto (b\mapsto \deg(ab))$ est un isomorphisme de~$A^k$
sur $(A^{r-k})^\vee$.

Soit $\ell$ un élément de $A^1_\R$ et soit $k$ un entier
tel que $k\leq r/2$.
L'application de Lefschetz associée à~$\ell$
est l'application  linéaire $\lambda^k\colon a\mapsto \ell^{r-2k} a$
de $A^k_\R$ dans~$A^{r-k}_\R$.
On définit aussi une forme bilinéaire symétrique $Q^k_\ell$ sur~$A^k_\R$
par
\[ Q^k_\ell (a,b)  = (-1)^k \deg( a\, \ell^{r-2k}b). \]
On note $P_\ell^k$ le sous-espace de~$A^k_\R$
formé des~$a\in A^k_\R$ tels que $\ell^{r+1-2k}a=0$.
Si $k>r/2$, on pose $P_\ell^k=0$.

On dit que $(A_\R,\ell)$ vérifie le théorème de Lefschetz difficile 
si $\lambda^k$ est un isomorphisme pour tout entier~$k$
tel que $0\leq k\leq r/2$.
Alors, pour tout entier~$k$ tel que $0\leq k\leq r/2$,
l'espace~$A^k_\R$ admet la \emph{décomposition de Lefschetz}:
\[ A^k_\R = P^k_\ell \oplus \ell P^{k-1}_\ell \oplus \dotsb \oplus \ell^k P_\ell^0\ ;\]
c'est une décomposition orthogonale pour la forme~$Q^k_\ell$.
Pour tout entier~$k$ tel que $r/2<k\leq r$, 
on obtient une décomposition similaire
en écrivant $A^k_\R=\ell^{2k-r} A^{r-k}_\R$:
\[ A^k_\R = \ell^{2k-r} P^{r-k}_\ell \oplus \ell^{2k-r-1} P^{r-k-1}_\ell \oplus \dotsb \oplus \ell^k P_\ell^0.\]

On dit que $(A_\R,\ell)$ vérifie les \emph{relations de Hodge--Riemann}
si la restriction à~$P^k_\ell$
de la forme quadratique associée à~$Q^k_\ell$ est définie positive
pour tout entier~$k$ tel que $0\leq k\leq r/2$.

\begin{exem}
La terminologie provient bien sûr des propriétés de l'algèbre
de cohomologie d'une variété complexe compacte (connexe) kählérienne.
Soit en effet $V$ une telle variété, soit $n$
sa dimension et soit $H(V)$
l'algèbre réelle graduée, gr-commutative, de cohomologie de De Rham.
On a $H^k(V)=0$ si $k>2n$.
La décomposition de Hodge munit $H(V)_\C$ d'une
bigraduation canonique $H^k(V)_\C=\smash{\bigoplus\limits_{p+q=k}} H^{p,q}(V)$;
on a $H^{p,q}(V)=0$ si $p$ ou $q$ n'appartient pas à l'intervalle~$[0,n]$.

On dispose d'un isomorphisme canonique $\int\colon H^{2n}(V)\to \R$.
Pour tout entier~$k$,
l'homomorphisme $a \mapsto (b\mapsto \int a\wedge b)$
induit un isomorphisme de $H^{k}(V)$ sur $H^{2n-k}(V)^\vee$
(dualité de Poincaré)
et de $H^{p,q}(V)$ sur $H^{n-p,n-q}(V)^\vee$.

Soit $\ell$ la classe dans $H^{2}(V)$ d'une forme de Kähler sur~$V$;
elle appartient à $H^{1,1}(V)$.
Alors, pour tout entier~$k$ tel que $k\leq n$,
l'application $a\mapsto \ell^k \wedge a$
de $H^{k}(V)$ dans $H^{k+2}(V)$  est injective
(théorème de Lefschetz difficile). On note alors $P^k(V)$
le sous-espace primitif de~$H^k(V)$,
noyau de $a\mapsto\ell^{k+1}\wedge a$.

La forme bilinéaire~$Q^k$ sur~$H^k(V)$ 
définie par $Q^k(a,b)=\int \ell^{n-k} \wedge a\wedge b$ 
est symétrique si $k$ est pair, alternée si $k$ est impair,
de sorte que la forme bilinéaire~$R^k$ sur $H^k(V)_\C$
définie par $R^k(a,b)=i^k Q^k(a,\overline b)$ est hermitienne
(forme de Riemann).
Soit $(p,q)$ un couple d'entiers tels que $p+q=k$.
La restriction à $P^k(V)_\C\cap H^{p,q}(V)$
de cette forme hermitienne est définie,
de signe $(-1)^{k(k-1)/2+q}$ (\emph{relations bilinéaires de Hodge--Riemann}).
Dans le cas particulier où~\mbox{$p=q$,} 
l'entier $k=2p$ est pair, le coefficient~$i^{k}$ 
qui intervient dans la définition de~$R^k$ est égal à~$(-1)^p$,
et sur $P^{2p}(V)_\C\cap H^{p,p}(V)$, la forme de Riemann
est définie positive.

Supposons de plus que $V$ soit projective
et notons $C(V)$ l'anneau des classes de cycles pour l'équivalence
homologique. L'homomorphisme
de classe de cycles $C(V)\to H(V)$  est alors injectif,
de sorte que $C(V)$ vérifie les conditions du paragraphe précédent.
Noter que lorsque $V$ est une variété torique projective lisse,
l'équivalence homologique coïncide avec l'équivalence rationnelle
et l'homomorphisme de classe de cycles est un isomorphisme.
\end{exem}

\begin{theo}\label{theo.hl}
Soit $M$ une géométrie combinatoire de rang~$>0$ et soit $r=\rk(M)-1$.
Soit $\ell\in \mathscr K_M$ une classe ample de~$\PL(M)$.
\begin{enumerate}\def\labelenumi{\theenumi)}
\item Le couple $(A(M),\deg)$ vérifie la dualité de Poincaré;
\item Le couple $(A(M)_\R,\ell)$ vérifie le théorème de Lefschetz difficile;
\item Le couple $(A(M)_\R,\ell)$ vérifie les relations de Hodge--Riemann.
\end{enumerate}
\end{theo}
C'est en quelque sorte le théorème principal de l'article
d'\citet*{adiprasito-huh-katz2015}. Nous donnerons
des indications de sa preuve dans la section suivante.
Voyons tout de suite comment il entraîne le théorème~\ref{theo.ahk}.

On commence par en déduire le corollaire suivant,
analogue aux inégalités de Khovanskii-Teissier.

\begin{coro}\label{coro.kt}
Soit $\alpha$ et $\beta$ des classes de~$A^1(M)_\R$.
Si $\alpha$ est nef, alors
\[ \deg(\alpha^{r-2}\beta^{2}) \deg(\alpha^r)
\leq \deg(\alpha^{r-1} \beta)^2. \]
\end{coro}
\begin{proof}
Par passage à la limite, il suffit de traiter le cas où $\alpha$
est ample.
La décomposition de Lefschetz $A^1(M)_\R=
P_\alpha^1(M) \oplus \langle \alpha \rangle$
est orthogonale pour la forme
$Q_\alpha^1$ définie par $Q_\alpha^1(x,y)=-\deg(x \alpha^{r-2} y)$,
laquelle est définie positive
sur~$P_\alpha^1(M)$ et définie négative sur~$\langle\alpha\rangle$.
L'inégalité à vérifier est évidente si $\beta$
est proportionnelle à~$\alpha$.
Sinon, le sous-espace $\langle\alpha,\beta\rangle$
est de dimension~$2$ et la restriction de la forme~$Q_\alpha^1$ 
y est de signature $(1,1)$. Son discriminant est donc négatif,
d'où le corollaire.
\end{proof}

\subsection{Démonstration du théorème~\ref{theo.ahk}}
On prouve d'abord que pour tout $k\in\{0,\dotsc,r\}$,
l'entier~$\mu^k(M)$ est strictement positif,
par exemple en déduisant du théorème de Weisner
qu'il est égal au nombre de « drapeaux initiaux descendants »
de longueur~$k$, c'est-à-dire de familles
$(P_1,\dotsc,P_k)$ de plats de~$M$ 
tels que $P_1\subset\dotsb\subset P_k$,
$\rk_M(P_j)=j$ pour tout~$j$,
et $\inf(P_1)>\inf(P_2)>\dotsb>\inf(P_k)>0$.
(Noter que ces conditions imposent $P_1\neq\emptyset$ et $P_k\neq\abs M$.)

D'après la proposition~\ref{prop.deg-muk},
on a $\mu^k(M)=\deg(\alpha_M^{r-k}\beta_M^k)$.
Ainsi, lorsque $k=r-1$, 
l'inégalité~(ii) du théorème~\ref{theo.ahk}
n'est autre que celle du corollaire~\ref{coro.kt},
appliquée à $\alpha=\alpha_M$ et $\beta=\beta_M$.

Sinon, on remplace~$M$ par la géométrie
combinatoire associée au matroïde tronqué $\tau(M)$
dont le treillis des plats est l'ensemble des plats de~$M$
dont le rang n'appartient pas à~$[k+1,r]$.
Son rang est égal à~$k+2$ et l'on 
a $\mu^j(M)=\mu^j(\tau(M))$ pour tout entier~$j$ tel que $j\leq k+1$.
L'inégalité voulue se déduit donc du cas déjà traité.

\section{Filtres}

\subsection{}
La démonstration du théorème~\ref{theo.hl} est combinatoire
et les paragraphes qui suivent ne font guère plus qu'en
décrire le cheminement; 
je renvoie à~\citep*[\S6--8]{adiprasito-huh-katz2015} pour les détails.

Cette démonstration consiste à partir
de l'éventail de l'espace projectif,
pour lequel la conclusion du théorème est évidente,
et à le modifier progressivement jusqu'à l'éventail~$\Sigma_M$,
de sorte qu'à chaque étape la conclusion du théorème reste vraie.
Ces modifications exigent d'introduire des éventails un peu plus compliqués.

On dira ainsi qu'un éventail~$\Sigma$ vérifie la dualité de Poincaré
en dimension~$r$ 
s'il existe un isomorphisme $\deg\colon A^r(\Sigma)\to\Z$
tel que l'anneau de Chow $(A(\Sigma),\deg)$ la vérifie.
On dira alors que $\Sigma$ vérifie le théorème de Lefschetz
difficile, resp. les relations de Hodge--Riemann, 
si $(A(\Sigma)_\R,\ell)$ les vérifie pour toute classe
ample $\ell\in\mathscr K_\Sigma$.

\subsection{}
Soit $M$ un matroïde sans boucle et  soit $\mathscr P_M$
le treillis de ses plats. On suppose que $M$ est de rang~$>0$
et on pose $r=\rk(M)-1$.

On note $\inf(\mathscr D)$ l'intersection, dans~$\abs M$,
des éléments d'un drapeau $\mathscr D$ de plats de~$M$; 
si $\mathscr D=\{P_1,\dotsc,P_d\}$,
où $\emptyset\subsetneq P_1\subsetneq\cdots\subsetneq P_d \subsetneq \abs M$,
on a donc $\inf(\mathscr D)=\abs M$ lorsque $d=0$, et $\inf(\mathscr D)=P_1$ sinon.

On dit que $I$ et~$\mathscr D$ sont compatibles, et on note $I<\mathscr D$,
si $I$ est une partie stricte de~$\inf(\mathscr D)$.
On dit qu'ils sont compatibles vis-à-vis de~$M$, et on note $I<_M\mathscr D$,
si $\Card(I)<\rk_M(\inf(\mathscr D))$; cela implique qu'ils sont compatibles,
car le rang d'un plat est majoré par son cardinal.

Soit $\mathscr D$ un tel drapeau et soit $I$ une partie de~$\inf(\mathscr D)$.
On note $\sigma_{I,\mathscr D}$ le cône engendré par
les vecteurs~$e_i$, pour $i\in I$, et les vecteurs~$e_P$, pour $P\in\mathscr D$.
Si $I$ et $\mathscr D$ sont compatibles, c'est un cône
de dimension~$\Card(I)+\Card(\mathscr D)$.

\subsection{}
On appelle \emph{filtre}\footnote{%
La terminologie est un peu trompeuse car on n'impose
pas l'hypothèse de stabilité par~$\wedge$;
dans le cas du treillis des parties d'un ensemble,
une telle partie n'est pas nécessairement un filtre au sens usuel.} 
de plats sur~$M$
une partie non vide~$\mathscr P$ de~$\mathscr P_M$,
ne contenant pas~$\emptyset$, 
qui contient tout plat de~$\abs M$ contenant un élément de~$\mathscr P$.

Soit $\mathscr P$ un filtre de plats de~$M$.
L'\emph{éventail de Bergman} associé à $(M,\mathscr P)$
est l'ensemble~$\Sigma_{M,\mathscr P}$
des cônes~$\sigma_{I,\mathscr D}$,
où $\mathscr D$ parcourt l'ensemble des drapeaux de
plats de~$M$ appartenant à~$\mathscr P$ et 
$I$ parcourt l'ensemble des parties de~$\abs M$ compatibles avec~$\mathscr D$
et telles que $\langle I\rangle \not\in\mathscr P$.
L'\emph{éventail de Bergman réduit} $\widetilde\Sigma_{M,\mathscr P}$
est l'ensemble de ces cônes, où l'on exige en outre
que $I$ et $\mathscr D$ soient compatibles vis-à-vis de~$M$.

Si $(I_1,\mathscr D_1)$ et $(I_2,\mathscr D_2)$ sont
des couples tels que $I_1<\mathscr D_1$ et $I_2<\mathscr D_2$,
alors $I_1\cap I_2 <\mathscr D_1\cap\mathscr D_2$,
et de même  pour la relation~$<_M$.
De plus, on a $\sigma_{I_1,\mathscr D_1}\cap \sigma_{I_2,\mathscr D_2}
= \sigma_{I_1\cap I_2,\mathscr D_1\cap\mathscr D_2}$.
Cela entraîne que $\Sigma_{M,\mathscr P}$ et $\widetilde \Sigma_{M,\mathscr P}$
sont effectivement des éventails de~$N_\R$.

Lorsque $\mathscr P=\mathscr P_M\setminus\{\emptyset\}$,
la condition $\langle I\rangle \not\in\mathscr P$ équivaut à $I=\emptyset$,
et la condition $I\subsetneq \inf(\mathscr D)$ est vérifiée
car $\inf(\mathscr D)$ n'est pas vide.
Dans ce cas, les deux éventails $\Sigma_{M,\mathscr P}$ 
et $\widetilde\Sigma_{M,\mathscr P}$ coïncident avec l'éventail~$\Sigma_M$.

Lorsque $M$ est le matroïde booléen sur~$\abs M$
(c'est-à-dire dont toute partie est plate), l'éventail $\Sigma_{M,\mathscr P}$
est l'éventail d'un polytope qui est obtenu par subdivisions étoilées
successives à partir d'un simplexe
\citep*[prop.~2.4]{adiprasito-huh-katz2015}.

\begin{lemm}\def\labelenumi{\theenumi)}
\begin{enumerate}
\item 
L'éventail $\Sigma_{M,\mathscr P}$ est un sous-éventail 
de l'éventail normal d'un polytope; en particulier, son cône
ample n'est pas vide.
\item
L'éventail $\widetilde\Sigma_{M,\mathscr P}$ est purement de dimension~$r$.
\end{enumerate}
\end{lemm}

\subsection{}
L'éventail $\Sigma_{M,\mathscr P}$
a pour rayons, d'une part, les vecteurs~$e_i$,
pour $i\in\abs M$ tel que $\langle i\rangle \not\in \mathscr P$,
et, d'autre part, les vecteurs~$e_{P}$, où $P\in\mathscr P\setminus\{\abs M\}$.
Son anneau de Chow $A(\Sigma_{M,\mathscr P})$, noté aussi $A(M,\mathscr P)$,
est ainsi le quotient
de l'anneau de polynômes $S_{M,\mathscr P}$ en des
indéterminées~$T_i$ (pour $i\in\abs M$) et $T_P$ (pour $P\in\mathscr P
\setminus\{\abs M\}$)
par l'idéal engendré par les éléments du type suivant:
\begin{enumerate}
\def\theenumi{$R_{\arabic{enumi}}$}\def\labelenumi{(\theenumi)}
\item $T_{P_1} T_{P_2}$, où $P_1,P_2\in\mathscr P$ 
sont deux plats incomparables;
\item $T_i T_P$, où $P\in\mathscr P$ et $i\in \abs M\setminus P$;
\item $\prod_{i\in I} T_i$, où $I$ est une partie libre non vide de~$\abs M$
telle que $\langle I\rangle \in\mathscr P$;
\item $(T_i+\sum_{P\ni i}T_P)-(T_j+\sum_{P\ni j}T_P)$,
où $i,j$ sont des éléments de~$\abs M$, distincts.
\end{enumerate}
Pour $i\in \abs M$ et $P\in\mathscr P$,
on notera $t_i$, resp. $t_P$, la classe de~$T_i$,
resp. de~$T_P$, dans l'anneau~$A(\Sigma_{M,\mathscr P})$.

L'éventail $\widetilde\Sigma_{M,\mathscr P}$  
est un sous-éventail 
de $\Sigma_{M,\mathscr  P}$. 
Son anneau de Chow $A(\widetilde\Sigma_{M,\mathscr P})$
est donc un quotient de l'anneau $A(\Sigma_{M,\mathscr P})$.
En vérifiant que les relations supplémentaires 
découlent de celles imposées dans $A(\Sigma_{M,\mathscr P})$,
on démontre la proposition suivante.

\begin{prop}
L'inclusion de $X(\widetilde\Sigma_{M,\mathscr P})$ dans
$X(\Sigma_{M,\mathscr P})$ induit un isomorphisme
de $A(\Sigma_{M,\mathscr P})$ sur $A(\widetilde\Sigma_{M,\mathscr P})$.
En particulier, $A^k(\Sigma_{M,\mathscr P})=0$ pour $k>r$.
\end{prop}

\begin{exem}
Supposons que $\mathscr P=\{\abs M\}$; pour simplifier les notations,
posons $\Sigma=\Sigma_{M,\{M\}}$
et $\widetilde\Sigma=\widetilde\Sigma_{M,\{M\}}$.

Les rayons de l'éventail~$\Sigma$ sont les vecteurs~$e_i$,
pour $i\in \abs M$.
Plus généralement, un cône~$\sigma_I$ appartient à~$\Sigma$
si et seulement si $\langle I\rangle \neq\abs M$;
il appartient à~$\widetilde\Sigma$ si et seulement si $\Card( I)\leq r$.
En posant $n=\Card(\abs M)-1$,
on voit donc que le support de~$\widetilde\Sigma$ est la réunion
des cônes de dimension~$\leq r$ de l'éventail de
l'espace projectif~$\P_n$.

L'anneau de Chow $A(M,\{M\})$ n'a donc 
aucun générateur de la forme~$T_P$;
les relations des types~($R_1$) et~($R_2$) sont alors triviales.
Les relations du type~($R_4$) entraînent que
les~$T_i$,  pour tout~$i\in\abs M$, ont même image dans~$A(M,\{\abs M\})$;
notons~$t$ cette image.
Les relations du type~($R_3$) s'écrivent alors $t^{r+1}=0$.
Autrement dit,
\[ A(M,\{\abs M\}) = \Z[t]/(t^{r+1}); \]
c'est donc l'anneau de Chow de l'espace projectif~$\P_r$,
et il vérifie de façon évidente la conclusion du théorème~\ref{theo.hl}.
\end{exem}

\subsection{}\label{ss.flip}
Ces notations introduites,
la démonstration du théorème~\ref{theo.hl}
procède par récurrence et
consiste à examiner le comportement
de l'anneau $A({M,\mathscr P})$
lorsqu'on adjoint au filtre~$\mathscr P$ un plat~$P$ 
qui est maximal dans~$\mathscr P_M\setminus \mathscr P$.
Dans l'article~\citep{adiprasito-huh-katz2015},
la modification d'éventails qui en résulte est appelée \emph{flip matroïdal}
de centre~$P$.
Cela revient à enlever à~$\Sigma_{M,\mathscr P}$
les cônes $\sigma_{I,\mathscr D}$
tels que
$I<\mathscr D$, $\langle I\rangle=P$ et $\inf(\mathscr D)\neq P$,
et à lui ajouter les cônes $\sigma_{I,\mathscr D}$
où $I<\mathscr D$, $\langle I\rangle\neq P$ et $\inf(\mathscr D)=P$.
Posons $\mathscr P'=\mathscr P\cup\{P\}$.
Le plat~$P$ est minimal dans~$\mathscr P'$.

Rappelons aussi que l'on note $M\contr P$ (resp. $M\restr P$) 
le matroïde dont les plats sont ceux de~$M$
contenant~$P$ (resp. contenus dans~$P$).

On raisonne par récurrence, en supposant la conclusion
du théorème~\ref{theo.hl} satisfaite par 
tout anneau de la forme $A(M_1,\mathscr P_1)$
tel que soit $\rk(M_1)< \rk(M)$,
soit $\rk(M_1)=\rk(M)$ et $\Card(\mathscr P_1) < \Card(\mathscr P)$.

\begin{prop}
Il existe un unique homomorphisme d'anneaux
\[ \Phi_P\colon A(M,\mathscr P)\to A(M,\mathscr P') \]
tel que $\Phi_P(t_Q)=t_Q$ pour tout $Q\in\mathscr P\setminus\{\abs M\}$
et tel que $\Phi_P(t_i)=t_i+t_P$ si $i\in P$, 
et $\Phi_P(t_i)=t_i$ sinon.
Si $\rk_M(P)=1$, c'est un isomorphisme.
\end{prop}

\begin{prop}
Soit $p$ un entier~$\geq 1$.
\begin{enumerate}\def\labelenumi{\theenumi)}
\item
Il existe un unique homomorphisme de groupes
$\Psi_P^p\colon A(M\contr P)\to A(M,{\mathscr P'})$
tel que $\Psi_P^p(t_{\mathscr D})=t_P ^p t_{\mathscr D}$
pour tout drapeau $\mathscr D$ de plats de~$M_P$;
il est homogène de degré~$p$.
\item
Il existe un unique homomorphisme de groupes
$\Gamma_P^p\colon A(M\restr P)\to A(M)$ 
tel que $\Gamma_P^p(t_{\mathscr D})=t_P^p t_{\mathscr D}$
pour tout drapeau $\mathscr D$ de plats de~$M^P$;
il est homogène de degré~$p$.
\end{enumerate}
\end{prop}

\begin{theo}\label{theo.dec}
L'homomorphisme
\[ \Phi_P + \sum_{p=1}^{\rk_M(P)-1} \Psi_P^p 
\colon A(M,\mathscr P) \oplus A(M\contr P)[-p] \to A(M,\mathscr P') \]
est un isomorphisme d'anneaux gradués,
où le symbole $[-p]$ signifie que la graduation est décalée de~$-p$.
\end{theo}

\begin{coro}
L'homomorphisme~$\Phi_P$ induit un isomorphisme
\[ A^r(M,\mathscr P) \xrightarrow\sim A^r(M,\mathscr P') \]
et l'algèbre $A(M,\mathscr P')_\R$ muni de l'isomorphisme
$\deg\circ\Phi_P\colon A^r(M,\mathscr P')\to\Z$,
vérifie la dualité de Poincaré.
\end{coro}

\subsection{}
La démonstration du théorème~\ref{theo.dec}
commence par établir la surjectivité
de l'homomorphisme indiqué. On en déduit ensuite que
l'homomorphisme $\Psi_P^{\rk_M(P)}$ induit un isomorphisme
de $A^{r-\rk_M(P)}(M\contr P)$ sur~$A^r(M,\mathscr P)$.
Sous l'hypothèse
que la dualité de Poincaré vaut pour $A(M,\mathscr P)$
et $A(M\contr P)$, une dernière étape prouve l'injectivité
de l'homomorphisme.
Compte tenu du corollaire, l'algèbre $A(M,\mathscr P')$ vérifie
alors la dualité de Poincaré.

Par récurrence,
cela prouve ainsi que pour tout matroïde~$M$ de rang~$r+1$ et tout filtre
de plats~$\mathscr P$ sur~$M$, l'anneau $A(M,\mathscr P)$
vérifie la dualité de Poincaré en dimension~$r$.
En particulier,  pour $\mathscr P=\mathscr P_M$,
l'anneau~$A(M)$ vérifie la dualité de Poincaré.

\subsection{}
La dualité de Poincaré acquise, la démonstration du
théorème de Lefschetz difficile et celle des inégalités
de Hodge--Riemann sont menées de front:
le théorème de Lefschetz difficile
affirme que la forme bilinéaire de Hodge--Riemann
est non dégénérée, et les inégalités de Hodge--Riemann
en précisent la signature.

Grâce à la remarque
que la signature est une fonction localement constante
sur l'espace des formes quadratiques non dégénérées,
un argument élémentaire de déformation 
prouve que si le théorème de Lefschetz difficile est vérifié
pour toute classe ample,
alors les inégalités de Hodge--Riemann valent
pour toute classe ample si et seulement si elle valent pour
\emph{une} classe ample.

Soit $\Sigma$ un éventail unimodulaire
vérifiant la dualité de Poincaré en dimension~$r$.
Pour tout rayon $v\in V_\Sigma$, 
l'algèbre quotient $A(\Sigma)_\R/\ann(t_v)$,
associée à l'étoile $\star_\Sigma(\sigma)$,
satisfait la dualité de Poincaré en dimension~$r-1$,
et
\citet*{adiprasito-huh-katz2015}
introduisent la  variante « locale » des 
inégalités de Hodge--Riemann qui postule
que cette algèbre vérifie ces inégalités
pour (l'image de) toute classe ample dans~$\mathscr K_\Sigma$.

\begin{prop}
Les inégalités de Hodge--Riemann locales impliquent
le théorème de Lefschetz difficile.
\end{prop}

\subsection{}
On démontre la validité
des inégalités de Hodge--Riemann pour l'algèbre $A(M,\mathscr P)_\R$
en raisonnant par récurrence d'abord sur le rang de~$M$ 
puis sur le cardinal de~$\mathscr P$. Lorsque $\mathscr P=\emptyset$,
on a déjà mentionné que l'algèbre $A(M,\mathscr P)_\R$,
isomorphe à l'algèbre $\R[t]/(t^{r+1})$ associée à~$\P_r$,
vérifie le résultat voulu.
Une première réduction permet de supposer que $M$ est une géométrie
combinatoire.

Revenons alors au contexte d'un flip matroïdal (\S\ref{ss.flip}):
$\mathscr P$ est un filtre de plats sur~$M$,
$P$ est un plat maximal dans~$\mathscr P_M\setminus \mathscr P$
et $\mathscr P'=\mathscr P\cup\{P\}$.

Tout d'abord, \citet*[prop.~3.5]{adiprasito-huh-katz2015} observent
que l'étoile $\star_{\Sigma_{M,\mathscr P'}}$
de tout rayon~$v$ est le produit 
des éventails $\Sigma_{M\restr P,\mathscr P\restr  P}$ et $\Sigma_{M\contr P}$
si le rayon~$v$ est associé à un plat~$P$,
et l'éventail $\Sigma_{M_i,\mathscr P_i}$ si le rayon~$v$
est associé à un élément~$i$ de~$\abs M$.
(On a noté $\mathscr P\restr P$ l'ensemble des plats de~$\mathscr P$
qui sont contenus dans~$P$; c'est un filtre de plats sur~$M\restr P$.)
Dans ce dernier cas, l'éventail vérifie les inégalités de Hodge--Riemann,
par l'hypothèse de récurrence. Dans le premier cas,
les deux éventails qui interviennent les vérifient également,
donc leur produit aussi: cela revient à prouver que l'algèbre
$A(M\restr P,\mathscr P\restr P)\otimes A(M\contr P)$ vérifie ces inégalités,
ce que l'on démontre en se ramenant au cas
du produit tensoriel $\R[t]/(t^{a+1})\otimes \R[t]/(t^{b+1})$
associé à $\P_a\times\P_b$.

Ainsi, l'éventail $\Sigma_{M,\mathscr P'}$ vérifie les 
inégalités locales de Hodge--Riemann. Il vérifie
donc le théorème de Lefschetz difficile.

Il reste à démontrer les inégalités de Hodge--Riemann,
mais il suffit maintenant de les vérifier pour \emph{une} classe.
La conclusion du théorème~\ref{theo.dec}
fournit un isomorphisme
de $\R$-espaces vectoriels gradués
\[ A(M,\mathscr P')_\R \simeq A(M,\mathscr P)_\R \oplus
       \big( \R[t]/(t^{\rk_M(P)-1}) \otimes_\R A(M\contr P) \big)[-1] \]
qu'\citet*{adiprasito-huh-katz2015}
utilisent pour construire des classes amples qui vérifient
les inégalités de Hodge--Riemann.

\section{Plats}

\subsection{}
Concluons ce rapport en évoquant une autre conjecture, encore ouverte,
en combinatoire énumérative des matroïdes, mais à laquelle
l'article de~\cite{huh-wang2017} apporte une réponse positive
dans le cas représentable.

Soit $M$ un matroïde. Pour tout entier~$k$, on note $M^{(k)}$
l'ensemble des plats de rang~$k$ de~$M$ et on pose $W_k(M)=\Card(M^{(k)})$;
ce sont les \emph{nombres de Whitney de seconde espèce} de~$M$.
Ils sont nuls pour $k<0$ ou $k>\rk(M)$.

\begin{conj}[conjecture de Rota--Welsh]
\label{conj.Wk}
Soit $M$ un matroïde et soit $r=\rk(M)$.
\begin{enumerate} \def\theenumi{\roman{enumi}}\def\labelenumi{(\theenumi)}
\item
La suite $(W_0(M),\dotsc,W_{r})$ est log-concave; en particulier,
elle est unimodale;
\item
Pour tout entier~$k$ tel que $0\leq k\leq r/2$, on a $W_k(M)\leq W_{r-k}(M)$;
\item
On a $W_0(M)\leq W_1(M)\leq \dotsb \leq W_{\lfloor r/2\rfloor}(M)$.
\end{enumerate}
\end{conj}
L'unimodalité a été conjecturée par \citet{rota1971}
et on doit à \citet{welsh1976} la suggestion que cette suite
pourrait être log-concave. \citet{mason1972}
conjecture même que le quotient $W_k(M)^2/W_{k-1}(M)W_{k+1}(M)$
serait toujours supérieur ou égal à $(k+1)/k$,
qui est la valeur prise par ce quotient
pour la structure de matroïde sur $\abs M$
pour laquelle toute partie est libre (« matroïde libre »).

\begin{theo}[\citealp{huh-wang2017}]
\label{theo.hw}
Soit $M$ un matroïde \emph{représentable} et soit $r=\rk(M)$.
Soit $p,q$ des entiers tels que $0\leq p\leq \inf(q,r-q)$.
Il existe une application injective $\phi\colon M^{(p)}\to M^{(q)}$
telle que $x\subset \phi(x) $ pour tout $x\in M^{(p)}$.
En particulier,  on a $W_p(M)\leq W_{q}(M)$.
\end{theo}
En prenant $p=k$ et $q=p+1$ (resp. $q=r-k$), 
on en déduit en particulier
que les assertions~(ii) et~(iii) de la conjecture~\ref{conj.Wk}
sont satisfaites pour un matroïde représentable.

Dans le cas d'un matroïde représentable de rang~$3$,
on retrouve  le théorème classique de~\cite{debruijn-erdos1948}
selon lequel $n$~points non alignés d'un plan projectif
déterminent au moins $n$~droites.

\subsection{}
Soit $M$ un matroïde, posons $n+1=\Card(\abs M)$ et $r=\rk(M)$;
on suppose que $\abs M=\{0,\dotsc,n\}$.
Supposons $M$ représentable sur un corps~$K$.
Notons $[x_0:\dotsb:x_n]$ les coordonnées homogènes de~$\P_{n,K}$
et identifions le complémentaire de l'hyperplan défini par $x_0=0$
à l'espace affine~$\A^n_K$.
Considérons un sous-espace affine~$L$ de dimension~$r$ de~$K^n$
dont l'adhérence~$X$ dans~$\P_{n,K}$ représente~$M$, au sens où
les hyperplans de coordonnées de~$\P_{n,K}$
découpent sur~$X$ un arrangement d'hyperplans qui représente~$M$.
Pour tout circuit~$C$ de~$M$, il existe
une famille $(a_{C,i})_{i\in C}$ d'éléments de~$K$, non tous nuls,
unique à multiplication près par un élément non nul de~$K$,
telle que $\sum_{i \in C} a_{C,i} x_i = 0$
sur~$X$.

Soit $Y$ l'adhérence de~$L$ dans le produit
des droites projectives~$(\P_{1,K})^n$.
Notons $[z_1,w_1], \dotsc,[z_n,w_n]$ les coordonnées
multi-homogènes de~$\P_{1,K}^n$.
La démonstration du théorème~\ref{theo.hw}
repose sur l'observation suivante, due à \citet{ardila-boocher2016} 
qui décrivent l'idéal homogène de~$Y$.

\begin{prop}
La variété~$Y$ est définie dans $(\P_{1,K})^n$
par la famille d'équations multi-homogènes
\[ \sum_{i\in C} a_{C,i} z_i \prod_{j\in C\setminus\{i\}} w_j = 0 , \]
où $C$ parcourt l'ensemble des circuits du matroïde~$M$.
\end{prop}

Pour tout plat~$P\in\mathscr P_M$, soit $Y_P$ l'intersection de~$Y$
et du sous-espace localement fermé de $(\P_{1,K})^n$
défini par $w_i=0$ si et seulement si $i\not\in P$.

\begin{coro}\label{coro.partition}
La famille $(Y_P)_{P\in\mathscr P_M}$ est une partition de~$Y$
en parties localement fermées.
De plus, pour tout plat~$P$ de~$M$, $Y_P$
est 
isomorphe à l'espace affine~$\A^{\rk_M(P)}_K$.
\end{coro}

\subsection{}
Bien que $Y$ soit singulière en général,
l'existence de la partition~$(Y_P)$
a des conséquences remarquables sur sa cohomologie.

Pour fixer les idées, choisissons un nombre premier~$\ell$
distinct de la caractéristique de~$K$
et notons $H(Y)$ la cohomologie étale de~$Y_{\overline K}$ à coefficients
dans~$\Q_\ell$. 
Lorsque $K$ est de  caractéristique~$p>0$,
on peut se ramener au cas où $K$ est une clôture
algébrique d'un corps fini.
Lorsque $K$ est de caractéristique zéro,
on peut préférer se ramener au cas où $K=\C$ 
et prendre pour $H(Y)$ la cohomologie
singulière de~$Y(\C)$ munie de sa structure de Hodge mixte.
Dans les deux cas, on dispose d'une filtration par le poids 
sur les espaces de cohomologie $H^k(Y)$.

Notons aussi $IH(Y)$ la \emph{cohomologie d'intersection} de~$Y$
\citep*{goresky-macpherson1983,beilinson-bernstein-deligne1982}:
disons juste que c'est la cohomologie d'un complexe 
borné à cohomologie constructible sur~$Y$,
le complexe d'intersection décalé~$\mathrm{IC}_Y[-r]$,
caractérisé (dans une catégorie dérivée convenable) 
d'une part par la propriété
qu'il prolonge le faisceau constant sur le lieu lisse de~$Y$, et d'autre part 
par les « conditions de perversité » (décalées) sur la dimension du support
de ses faisceaux de cohomologies et de ceux de son dual de Verdier.

Les espaces $H^k(Y)$ et $IH^k(Y)$ sont nuls si $k\not\in[0,2r]$. De plus,
\citep[th.~3.1]{bjorner-ekedahl2009} déduisent du corollaire~\ref{coro.partition} que pour tout entier~$k$,
\begin{enumerate}\def\labelenumi{\theenumi)}
\item
Si $k$ est impair, on a $H^k(Y)=0$;
\item
Si $k$ est pair, 
l'espace $H^k(Y)$ est pur de poids~$k$
et de dimension $W_{k/2}(M)$,
engendré par les classes de cycles $[\overline{Y_P}]$,
pour $P\in M^{(k/2)}$;
\item
L'homomorphisme canonique de~$H^k(Y)$ dans $IH^k(Y)$ est injectif.
\end{enumerate}
Par récurrence sur le cardinal de~$\mathscr P_M$,
les deux premières assertions se déduisent de la cohomologie des espaces
affines et de la suite exacte longue de cohomologie à support
compact associée à une partition de~$Y$ en un ouvert
et le fermé complémentaire.
La troisième assertion
est  démontrée par~\citet[th.~2.1]{bjorner-ekedahl2009}
lorsque $K$ est la clôture algébrique d'un corps fini,
et par~\citet[th.~1.8]{weber2004b} lorsque $K=\C$:
le noyau de l'homomorphisme canonique de $H^k(Y)$
dans~$IH^k(Y)$ est la partie de poids~$<k$ de~$H^k(Y)$.
C'est une conséquence du comportement des poids
par les six opérations cohomologiques usuelles
\citep{deligne1974b,deligne1980}
et du fait que le prolongement intermédiaire d'un faisceau
pervers pur est un faisceau pervers pur de même poids.
 
\subsection{}
La cohomologie d'intersection $IH(Y)$ de~$Y$
est un module sur sa cohomologie usuelle $H(Y)$.
En particulier, pour tout fibré en droites ample~$\mathscr L$
sur~$Y$ et tout entier~$k$ tel que $0\leq k\leq r/2$,
on peut considérer l'homomorphisme de Lefschetz
\[  c_1(\mathscr L)^{r-2k}\cap \colon IH^{2k}(Y)  \to IH^{2r-2k}(Y). \]
Cet homomorphisme est injectif: 
le théorème de Lefschetz difficile
vaut pour la cohomologie d'intersection
\citep*[5.4.10, 6.2.10]{beilinson-bernstein-deligne1982}.

Puisque l'homomorphisme canonique 
de $H^{2k}(Y)$ dans $IH^{2k}(Y)$ est injectif,
il en résulte que l'homomorphisme de Lefschetz
\[  c_1(\mathscr L)^{r-2k}\cap \colon H^{2k}(Y)  \to H^{2r-2k}(Y) \]
est encore injectif. Autrement dit: le théorème de Lefschetz
difficile vaut pour la cohomologie usuelle de~$Y$.
Cela démontre déjà l'inégalité $W_k(M)\leq W_{r-k}(M)$ 
pour tout entier~$k$
tel que $0\leq k\leq r/2$.
Si $p,q$ sont des entiers tels que $0\leq p\leq \inf(q,r-q)$, l'homomorphisme
$c_1(\mathscr L)^{q-p}\cap\colon H^{2p}(Y)\to H^{2q}(Y)$
est a fortiori injectif, 
ce qui entraîne l'inégalité $W_p(M)\leq W_{q}(M)$
du théorème~\ref{theo.hw}.

\subsection{}
Prenons pour fibré en droites~$\mathscr L$ le produit
tensoriel externe des fibrés $\mathscr O(1)$ sur les $n$~facteurs
de $(\P_{1,K})^n$
et considérons
la matrice~$\mathit\Lambda=(\mathit\Lambda_{P,Q})$ 
de l'application $c_1(\mathscr L)^{q-p}\cap$
dans les bases 
$([\overline {Y_P}])_{P\in M^{(p)}}$ de~$H^{2p}(Y)$
et 
$([\overline {Y_Q}])_{Q\in M^{(q)}}$ de~$H^{2q}(Y)$.

L'intersection
$ c_1(\mathscr L)\cap [\overline {Y_P}]$
est somme des termes $[\overline{Y_{\langle P,i\rangle}}]$
pour $i\in\{1,\dotsc,n\}\setminus P$.
Par suite, on a $\mathit\Lambda_{P,Q}=0$
si $P\not\subset Q$.
D'après ce qui précède, la matrice~$\mathit\Lambda$
est de rang maximal~$W_p(M)$,
donc au moins  un mineur de taille $W_p(M)$
en est inversible.
Ce mineur est de type~$P\times P'$, où $P'$ est une partie de~$Q$
de cardinal $W_p(M)$;
une fois choisi un ordre total sur~$P$ et~$P'$, ce mineur
se développe comme une somme de termes de la forme
\[   \pm \prod_{P\in M^{(p)}} \mathit\Lambda_{P, \iota(P)} \]
où $\iota$ parcourt l'ensemble des bijections de~$P$ sur~$P'$.
Au moins l'un de ces termes n'est pas nul: il correspond
à une application injective $\iota\colon M^{(p)}\to M^{(q)}$
telle que $P\subset \iota(P)$  pour tout $P\in M^{(p)}$.
Cela conclut la preuve du théorème~\ref{theo.hw}.

\subsection{}
Dans le but de généraliser ce théorème aux matroïdes non représentables,
on peut observer que la cohomologie de~$Y$ possède un modèle combinatoire,
décrit uniquement en termes du matroïde~$M$.

Soit donc $M$ un matroïde, non nécessairement représentable.
Notons $B(M)$ le groupe abélien libre sur l'ensemble des plats de~$M$,
et soit $(\delta_P)_{P\in\mathscr P_M}$ sa base canonique.
On munit $B(M)$ de la graduation pour laquelle, pour tout entier~$k$,
$B^k(M)$ est le sous-groupe engendré par les fonctions~$\delta_P$,
où $P$ parcourt l'ensemble $M^{(k)}$ des plats de rang~$k$.
On définit ensuite une multiplication dans~$B(M)$ par
\[ \delta_P\cdot \delta_Q = \begin{cases}
 \delta_{P\vee Q} & \text{si $\rk_M(P)+\rk_M(Q)=\rk_M(P\vee Q)$; } \\
0 & \text{sinon.} \end{cases} \]
Pour ces structures, $B(M)$ est une algèbre commutative graduée,
que \cite{huh-wang2017} appellent l'\emph{algèbre de Möbius graduée}
du matroïde~$M$.

Posons aussi $\lambda = \sum\limits_{i\in \abs M} \delta_i\in B^1(M)$.

Lorsque $M$ est représentable, l'unique homomorphisme
d'espaces vectoriels gradués de~$B(M)_{\Q_\ell}$ dans~$H(Y)$ 
qui applique~$\delta_P$ sur~$[\overline{Y_P}]$ est un isomorphisme
d'algèbres graduées, et l'application de Lefschetz $c_1(\mathscr L)\cap$
correspond à la multiplication par~$\lambda$.
\cite{huh-wang2017}
conjecturent alors que pour tout matroïde~$M$ de rang~$r$,
non nécessairement représentable,
l'application $\lambda^k\colon B^k(M)\to B^{r-k}(M)$
est injective
pour tout entier~$k$ tel que $0\leq k\leq r/2$.
 La conclusion du théorème~\ref{theo.hw} s'en déduirait
alors comme ci-dessus.

\bibliographystyle{mynat}
 \bibliography{aclab,acl,logconcave}

\end{document}